\documentclass[11pt]{article}

\usepackage{mathtools, amsthm, accents, tikz, amssymb, latexsym, amsmath, bm, amscd, amsfonts, array, lmodern, enumerate, stmaryrd, rotating, caption, graphicx, hyperref, float,tikz-cd, color, yfonts}
\usepackage[new]{old-arrows}
\usepackage[outline]{contour}
\usetikzlibrary{calc}
\usepackage[all]{xy}
\CompileMatrices
\hypersetup{nesting=true,debug=true,naturalnames=true}

\def\C{\protect\operatorname{Conf}}
\def\UC{\protect\operatorname{UConf}}
\def\UD{\emph{UD\hspace{.2mm}}}
\def\UDT{\emph{UDT\hspace{.2mm}}}

\newtheorem{proposition}{Proposition}[section]
\newtheorem{corollary}[proposition]{Corollary}
\newtheorem{definition}[proposition]{Definition}
\newtheorem{theorem}[proposition]{Theorem}
\newtheorem{remark}[proposition]{Remark}
\newtheorem{example}[proposition]{Example}
\newtheorem{lemma}[proposition]{Lemma}

\begin{document}

\title{Borsuk-Ulam property for graphs}

\author{Daciberg Lima Gon\c{c}alves and Jes\'us Gonz\'alez}

\date{\today}

\maketitle

\begin{abstract}
For finite connected graphs $\Gamma$ and $G$, with $\Gamma$ admitting a free involution $\tau$, we characterize the based homotopy classes $\alpha\in[\Gamma,G]$ for which the Borsuk-Ulam property holds in the sense of Gon\c{c}alves, Guaschi and Casteluber-Laass, i.e., the homotopy classes $\alpha$ so that each of its representatives $f\in\alpha$ satisfies $f(x) = f(\tau\cdot x)$ for some $x\in\Gamma$. This is attained through a graph-braid-group perspective aided by the use of discrete Morse theory.
\end{abstract}



\section{Introduction and main result}\label{intro}
In its classical formulation, the Borsuk-Ulam Theorem asserts that, for any continuous map
\begin{equation}\label{classicaln}
f\colon S^n\to \mathbb{R}^n,
\end{equation}
there is a point $x\in S^n$ so that both $x$ and its antipodal $-x$ have the same image under $f$. Such a phenomenon has been intensively studied in the last 15 years within generalized contexts, namely, for maps $f\colon M\to N$ between spaces $M$ and $N$, where $M$ admits a free involution. For instance, the case where $M$ ranges over surfaces or suitable families of 3-manifolds is now reasonably well understood \cite{MR3614297,BGH1,MRBGH,BauHaGoZv,MR3619753,MR4431413,MR2209795,MR2840097}. The case where $N$ has non-trivial homotopy information leads to a more refined problem, as the Borsuk-Ulam question can then have different answers for different homotopy classes\footnote{Unless otherwise noted, spaces are assumed to come equipped with base points which must be preserved by maps between spaces. Likewise, homotopy classes are meant in the based sense.} in $[M,N]$.

\begin{definition}[\cite{MR3947929}]
Assume $M$ admits a free involution $\tau$. We say that the Borsuk-Ulam property holds for a homotopy class $\alpha\in[M,N]$ if for every representative $f\in\alpha$ there is a point $x\in M$ such that $f(x)=f(\tau\cdot x)$. If the above condition holds for all homotopy classes in $[M,N]$, we say that the triple $(M,\tau,N)$ satisfies the Borsuk-Ulam property.
\end{definition}

We give a complete answer to the Borsuk-Ulam problem in the case where both $M$ and $N$ are 1-dimensional compact connected objects. We will thus focus on maps $f\colon \Gamma \to G$ between finite connected graphs $\Gamma$ and~$G$, addressing the Borsuk-Ulam property with respect to some fixed free involution $\tau$ on $\Gamma$.

\begin{remark}\label{notaninterval}{\em
In the classical situation~(\ref{classicaln}) with $n=1$, the circle plays no essential role. Indeed, by considering the differences $f(x)-f(\tau\cdot x)$, it can be seen that any map $f\colon \Gamma\to \mathbb{R}$ satisfies the Borsuk-Ulam property. Such a pleasant situation changes drastically when $\mathbb{R}$ (or an interval, for that matter) is replaced by a more general graph~$G$ which, in what follows, will be assumed not to be homeomorphic to an interval. In particular the configuration spaces $\C_2(G)$ and $\UC_2(G)$ consisting respectively of pairs $(x_1,x_2)$ and of subsets $\{x_1,x_2\}$ with $x_1\neq x_2$ are both connected.
}\end{remark}

The Borsuk-Ulam property for $(\Gamma,\tau,G)$ as above is described next. 
\begin{theorem}\label{maintheorem}
If $G$ is not homeomorphic to a circle or to an interval, then the Borsuk-Ulam property fails for all homotopy classes in $[\Gamma,G]$, i.e., for every $\alpha\in[\Gamma,G]$ there is a representative $f\in\alpha$ satisfying $f(x)\neq f(\tau\cdot x)$ for all $x\in\Gamma$.
\end{theorem}

When $G$ is a circle, the behavior of the Borsuk-Ulam property sits in between Remark~\ref{notaninterval} and Theorem~\ref{maintheorem}. The explicit answer, given in Theorem~\ref{secondmain} below, generalizes~\cite[Proposition~6]{MR3947929} and depends on the Euler characteristic $\chi(\Gamma)$. The latter number is even, in view of the free involution~$\tau$, and at most 0, since $\Gamma$ is connected. Say $\chi(\Gamma)=-2m$ with $m\geq0$. 

\begin{theorem}\label{secondmain}
If $G$ is homeomorphic to a circle $S^1$, then the Borsuk-Ulam property holds for most of the homotopy classes in $[\Gamma,S^1]$. Explicitly, under a certain identification of $[\Gamma,S^1]$ with $\mathbb{Z}^{2m+1}$, the homotopy classes of maps $\Gamma\to S^1$ for which the Borsuk-Ulam property fails are precisely the $(2m+1)$-tuples $(p,p_1,p_1,p_2,p_2,\ldots,p_m,p_m)$ with $p$ odd (and  $p_1,\ldots,p_m$ arbitrary).
\end{theorem}

Observe that Theorems~\ref{maintheorem} and~\ref{secondmain} can be stated replacing based homotopy classes by free homotopy classes. This is clear in the case of Theorem \ref{maintheorem}, while the case of Theorem~\ref{secondmain} follows from the fact that $\pi_1(S^1)$ is abelian, so that based homotopy classes and free homotopy classes coincide.

As in \cite{MR3993193,MR2639841,MR01,MR3947929,MR4235703}, we study the Borsuk-Ulam property for graphs through a sharp algebraic model in terms of braid groups. In our case (graphs), the critical information comes from the detailed control of the topological combinatorics associated to graph configuration spaces (both in the ordered and unordered contexts) provided by Farley-Sabalka's discrete gradient field on Abrams' homotopy model. All needed details are reviewed in Section~\ref{sectiononDMT}. Section~\ref{sectiononBUP} is devoted to the proof of Theorems~\ref{maintheorem} and~\ref{secondmain}.

\smallskip\noindent {\bf Acknowledgements:} The first author was partially supported by the FAPESP ``Projeto Tem\'atico-FAPESP Topologia Alg\'ebrica, Geom\'etrica e Diferencial'' 2016/24707-4 (S\~ao Paulo-Brazil). The potentiality of this paper was first realized by the authors during the workshop ``Topological Complexity and Motion Planning 22w5182'', held in Oaxaca, Mexico, from May 29 to June 3, 2022. Both authors are grateful to CMO-BIRS for a rich and stimulating workshop environment. The main part of this work was done during the visit of the first author to the Mathematics Department of Cinvestav, from September 7 to September 23, 2022. The first author is greatly thankful for the invitation and the hospitality of the institute.

\section{Graph braid groups via discrete Morse theory}\label{sectiononDMT}
We start by collecting the ingredients we need about Forman's discrete Morse theory and Farley-Sabalka's gradient field on Abrams' discrete model for (ordered and unordered) graph configuration spaces. For details, the reader is referred to (\cite{MR2701024,jorgetereyo,MR2171804,MR1358614,MR1926850,MR2833585}.

\subsection{Discrete Morse theory}\label{subsectionDMT}
Let $X$ be a connected finite regular CW complex with cell poset $\mathcal{F}$ partially ordered by inclusion\footnote{Cells are meant in the closed sense.}. For a cell $a\in\mathcal{F}$, we use $a^{(p)}$ as a shorthand of $\dim(a)=p$. The Hasse diagram of $\mathcal{F}$, $H_\mathcal{F}$, is thought of as a directed graph with arrows $a^{(p+1)}\searrow b^{(p)}$ oriented from the higher dimensional cell to the lower dimensional cell. Let $W$ be a partial matching on $H_\mathcal{F}$, i.e., a directed subgraph of $H_\mathcal{F}$ all whose vertices have degree 1. The modified Hasse diagram $H_\mathcal{F}(W)$ is obtained from $H_\mathcal{F}$ by reversing all arrows of $W$. A reversed edge is denoted as $b^{(p)}\nearrow a^{(p+1)}$, in which case $a$ is said to be collapsible and $b$ is said to be redundant. A path $\lambda$ in $H_\mathcal{F}(W)$ is a chain of up-going and down-going arrows $$a_0\nearrow b_1\searrow a_1\nearrow\cdots\nearrow b_k\searrow a_k.$$ The path $\lambda$ is said to be a cycle when $a_0=a_k$. If there are no cycles, $W$ is called a gradient field and, in such a case, cells of $X$ that are neither redundant nor collapsible are said to be critical. In what follows we assume that $W$ is a gradient field on $X$.

The subgraph of the one skeleton $X^{(1)}$ consisting of all vertices and of all collapsible edges forms a maximal forest $F_X$ with as many components as there are critical 0-cells in $X$. Add critical edges as needed in order to form a maximal tree $T_X$. Fix a vertex $v_0\in X^{(0)}$ as base point. Collapsing $T_X$ to $v_0$ yields a generating set $\{\beta_e\}_{e}$ for the fundamental group $\pi_1(X;v_0)$, where $e$ runs over the set of (arbitrarily oriented) critical 1-cells of $X$ that are not part of $T_X$. Explicitly, for each vertex $u\in X^{(0)}$, let $\beta_u$ be the unique path in $T_X$ determined by the ordered sequence of non repeating edges connecting $v_0$ to~$u$. Then, for a critical 1-cell $e$ from $u_1$ to $u_2$ that is not part of $T_X$, the loop
\begin{equation}\label{representingloop}
\beta_{u_1}\star e\star\beta_{u_2}
\end{equation}
represents the homotopy class~$\beta_e\in\pi_1(X;v_0)$, which will simply be denoted as $e\in\pi_1(X;v_0)$ ---the context clarifies whether we refer to the actual cell or to the corresponding homotopy class. Farley and Sabalka go further describing a set of relations among the homotopy generators that yield a presentation of $\pi_1(X,v_0)$. The relations depend on the critical 2-cells and the redundant 1-cells. We omit the details as we will not have need to use the relations.

\subsection{Farley-Sabalka gradient field on Abrams model}\label{subsectionFSA}
Let $G$ be a finite connected graph. By inserting a few non-essential vertices, we can assume $G$ is simplicial, i.e., that $G$ contains no loops nor double edges. Let $\C_2(G)$ denote the ordered configuration space of pairs $(x,y)\in G^2$ with $x\neq y$, and let $\UC_2(G)$ denote the orbit space by the involution $(x,y)\mapsto(y,x)$. Abrams homotopy model $D_2(G)$ for $\C_2(G)$ is the subcomplex of $G\times G$ whose cells are the ordered pairs $c=(c_1,c_2)$ of cells\footnote{A cell of $G$ is either a vertex or a (closed) edge.} of $G$ with $c_1\cap c_2=\varnothing$. The orbit complex $\UD_2(G)$ resulting from the involution $(c_1,c_2)\mapsto(c_2,c_1)$ is the corresponding homotopy model for $\UC_2(G)$. Thus, cells of $\UD_2(G)$ are sets $c=\{c_1,c_2\}$ of  disjoint cells $c_i$ of $G$. Both in the ordered and unordered settings:
\begin{itemize}
\item the cells $c_1$ and $c_2$ are called the ingredients of $c$;
\item the dimension of $c$ is the sum of the dimensions of $c_1$ and $c_2$;
\item the orientation of a 1-dimensional cell with vertex ingredient $v$ and edge ingredient $(v_1,v_2)$ (so $u\neq v_1<v_2\neq u$) will be inherited from that of $(v_1,v_2)$. For instance, the ordered 1-cell $((v_1,v_2),u)$ is oriented from $(v_1,u)$ to $(v_2,u)$.
\end{itemize}

The construction of Farley-Sabalka gradient field on $D_2(G)$ and of its quotient on $\UD_2(G)$ require some preliminary notation. Start by choosing a maximal tree $T$ of $G$. Edges of $G$ outside $T$ are called \emph{deleted edges}. Fix a planar embedding of $T$ and a root of $T$ (i.e., a vertex of degree 1 in~$T$), which is denoted by~$0$. The rest of the vertices of $G$ are consecutively numbered $1,2,\ldots$ as we first find them in the walk along $T$ that starts at $0$ and that takes the leftmost branch at any given intersection, turning around when a vertex of degree one is reached. An edge $e$ bounded by vertices $u$ and $v$ with $u<v$ is denoted by $e=(u,v)$, and is oriented from $u$ to $v$. Under such conditions we also write $v=\sigma(e)$, the \emph{source} of $e$, and $u=\tau(e)$, the \emph{target} of $e$. The source-target notation is compatible with the fact that, by collapsing $T$ down to its root $0$, we can think of $\pi_1(G,0)$ as the free group generated by the deleted edges (each with the vertex-ordering orientation). Note that if $(u,v)$ is non-deleted, $u$ is determined as the vertex adjacent to $v$ in $T$ that is located in the $T$-path leading from $v$ back to $0$, so we can safely write $e_v=(u,v)$. In particular, non-deleted edges $e_v$'s will be ordered according to the order of the corresponding $v$'s.

Let $c=(c_1,c_2)$ (or $c=\{c_1,c_2\}$) be a cell of $D_2(G)$ (or of $\UD_2(G)$). A vertex ingredient $v=c_i$ of~$c$ is said to be critical in $c$ if either $v=0$ or, else, if replacement of $v$ by $e_v$ in $c$ fails to yield a cell of $D_2(G)$ (or of $\UD_2(G)$). Likewise, and edge ingredient $e=c_j$ of $c$ is said to be critical in $c$ if either $e$ is deleted or, else, if $e=e_v$ and there is a vertex ingredient $u$ of $c$ adjacent to $\tau(e)$ with $\tau(e)<u<v$. With such notation, $c$ is critical in Farley-Sabalka gradient field provided $c_1$ and $c_2$ are both critical in $c$. Otherwise, if the smallest\footnote{With respect to the vertex-edge ordering discussed in the previous paragraph.} of the non-critical ingredients $c_1$ and $c_2$ is:
\begin{enumerate}[(i)]
\item a vertex $v$, then $c$ is redundant and $c\nearrow d$, where $d$ is obtained from $c$ by replacing $v$ by $e_v$; 
\item an edge $e_v$, then $c$ is collapsible and $d\nearrow c$, where $d$ is obtained from $c$ by replacing $e_v$ by $v$.
\end{enumerate}
In other words, the only critical 0-cells in $D_2(G)$ are $(0,1)$ and $(1,0)$, while $\{0,1\}$ is the only critical $0$-cell in $\UD_2(G)$. All other $0$-cells are redundant. Likewise, for a $1$-cell $c$ with vertex ingredient~$u$ and edge ingredient $e$ we have:
\begin{enumerate}[(i)]\addtocounter{enumi}{2}
\item If $e$ is deleted, then $c$ is critical if $u$ is critical in $c$, otherwise $c$ is redundant.
\item If $e$ is non-deleted, say $e=e_v$, and
\begin{itemize}
\item either $u=0$ or $v<u$, then $c$ is collapsible;
\item $0<u<v$ with $u$ critical in $c$ (in which case $\tau(e_u)=\tau(e_v)$ is forced), then $c$ is critical;
\item $0<u<v$ with $u$ non-critical in $c$, then $c$ is redundant.
\end{itemize}
\end{enumerate}
Lastly, a 2-cell $c$ is critical if both of its ingredients are deleted, otherwise $c$ is collapsible.

\subsection{Graph braid groups}\label{secciongbg}
We now recover the assumption in Remark~\ref{notaninterval} about $G$ not being homeomorphic to an interval, so we can use the facts reviewed above in order to describe generators for $$P_2(G)=\pi_1(D_2(G))=\pi_1(\C_2(G))\quad\text{and}\quad B_2(G)=\pi_1(\UD_2(G))=\pi_1(\UC_2(G)),$$
the (pure and full, respectively) braid groups of two (ordered and unordered, respectively) non-colliding particles in $G$. In particular, the notation set up in Subsections~\ref{subsectionDMT} and~\ref{subsectionFSA} will be in effect throughout the rest of the paper.

The case of $B_2(G)$ is slightly easier as $\UD_2(G)$ has a single critical 0-cell, so that the 0-cells and the collapsible 1-cells of $\UD_2(G)$ span a maximal tree $\UDT$ of the 1-skeleton of $\UD_2(G)$. Thus, after collapsing $\UDT$ to its base point, which is taken to be the critical 0-cell $\{0,1\}$, we get that the critical 1-cells $\{u,e\}$ of $\UD_2(G)$ yield corresponding generators $\{u,e\}\in B_2(G)$.

Since $D_2(G)$ has two critical 0-cells, we need to identify the two components of the maximal forest $DF$ determined by the 0-cells and the collapsible 1-cells of $D_2(G)$. Recall that any collapsible 1-cell~$c$ of $D_2(G)$ with vertex ingredient $u$ and edge ingredient $e_v$ satisfies $u=0$ of $v<u$. In particular, $c$ joins 0-cells $(w_1,w_2)$ and $(w_1',w_2')$ satisfying $w_1<w_2$ if and only if $w_1'<w_2'$. Consequently, $DF$ consists of two trees $DT_u$ and $DT_d$, where the vertices $(w_1,w_2)$ in the former (latter) tree satisfy $w_1<w_2$ ($w_1>w_2$). Now, as reviewed above, we need to add a single critical 1-cell of $D_2(G)$ to $DF$ in order to get a maximal tree $DT$. For our purposes, the required critical 1-cell will have the form $(a,(b,c))$ with $b<a<c$. (Note that such an edge goes from $(a,b)\in DT_d$ to $(a,c)\in DT_u$.) The explicit values of $a,b,c$ will be spelled out later, depending of the actual graph $G$. All other critical 1-cells $(r,(s,t))$ and $((s,t),r)$ of $D_2(G)$ yield corresponding generators $(r,(s,t)),((s,t),r)\in P_2(G)$ after collapsing $DT$ to its base point, which is now taken to be $(0,1)$.  

\begin{proposition}\label{inclusion}
The inclusion $\iota\colon P_2(G)\hookrightarrow B_2(G)$ induced by the 2-fold covering $D_2(G)\to\UD_2(G)$ is determined on generators by
\begin{align*}
\iota(r,(s,t))=\begin{cases}
\{r,(s,t)\}, & \mbox{if \  $r<s<t$;} \\ 
\{a,(b,c)\}^{-1} \cdot \{r,(s,t)\}, & \mbox{if \  $s<r<t$;} \\ 
\{a,(b,c)\}^{-1} \cdot \{r,(s,t)\} \cdot \{a,(b,c)\}, & \mbox{if \  $s<t<r$,} 
\end{cases}\\
\iota((s,t),r)=\begin{cases}
\{a,(b,c)\}^{-1} \cdot \{r,(s,t)\} \cdot \{a,(b,c)\}, & \mbox{if \  $r<s<t$;} \\
\{r,(s,t)\} \cdot \{a,(b,c)\}, & \mbox{if \  $s<r<t$;} \\ 
\{r,(s,t)\}, & \mbox{if \  $s<t<r$.}
\end{cases}
\end{align*}
In particular, $\iota((b,c),a)=\{a,(b,c)\}^2$.
\end{proposition}
\begin{proof}
Since the 2-fold covering projection $D_2(G)\to\UD_2(G)$ is cellular and preserves the critical/col\-laps\-ible/redundant nature of cells, the effect of the induced monomorphism $\iota$ is easily readable in terms of the generators describe above. Here we only consider the case of $\iota((s,t),r)$ with $r<s<t$, assumption in force throughout the rest of the proof, leaving to the reader the fully parallel details in the other five cases.

As a path in $D_2(G)$, the edge $((s,t),r)$ goes from $(s,r)\in DT_d$ to $(t,r)\in DT_d$. Then the paths $\beta_{(s,r)}$ and $\beta_{(t,r)}$ in~(\ref{representingloop}) are given by
$$
\beta_{(s,r)}=\gamma^{(0,1)}_{(a,c)}\star(a,(b,c))^{-1}\star\delta^{(a,b)}_{(s,r)}\mbox{ \ \ \ \ and \ \ \ \ } 
\beta_{(t,r)}=\gamma^{(0,1)}_{(a,c)}\star(a,(b,c))^{-1}\star\delta^{(a,b)}_{(t,r)},
$$
where $\gamma$-paths and $\delta$-paths consist of collapsible cells. Explicitly, $\gamma^{(0,1)}_{(a,c)}$ is the unique simple path in $DT_u$ connecting $(0,1)$ to $(a,c)$, while $\delta^{(a,b)}_{(s,r)}$ and $\delta^{(a,b)}_{(t,r)}$ are the unique simple paths in $DT_d$ connecting $(a,b)$ to $(s,r)$ and $(t,r)$, respectively.
$$
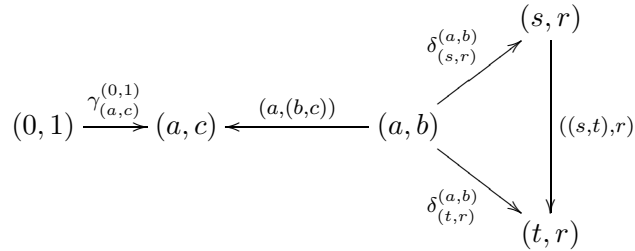
\xymatrix{
& & & & (s,r) \ar[dd]^{((s,t),r)}\\
(0,1) \ar[r]^{\gamma^{(0,1)}_{(a,c)}} & (a,c) && (a,b) \ar[ll]_{(a,(b,c))} \ar[ur]^{\delta^{(a,b)}_{(s,r)}} \ar[dr]_{\delta^{(a,b)}_{(t,r)}}\\
& & & & (t,r)}$$
The loop~(\ref{representingloop}) representing $((s,t),r)\in P_2(G)$ is then
\begin{equation}\label{representingloopexpanded}
\gamma^{(0,1)}_{(a,c)}\star(a,(b,c))^{-1}\star\delta^{(a,b)}_{(s,r)}\star((s,t),r)\star
\left( \gamma^{(0,1)}_{(a,c)}\star(a,(b,c))^{-1}\star\delta^{(a,b)}_{(t,r)} \right)^{-1},
\end{equation}
and the asserted expression for $\iota((s,t),r)$ now follows by noticing that the portions corresponding to $\gamma$-paths and $\delta$-paths are sent by $\iota$ into $UDT$ and, so, get squeezed to the base point $\{0,1\}$.
\end{proof}

From this point on we will think of $P_2(G)$ as a honest subgroup of $B_2(G)$, omitting to write the symbol $\iota$ when thinking of an element in $P_2(G)$ as an element of $B_2(G)$.

\begin{corollary}\label{conjugacion}
For a critical 1-cell $e$ of $D_2(G)$ other than $(a,(b,c))$, the conjugate
$$
\{a,(b,c)\}\cdot e\cdot\{a,(b,c)\}^{-1}\in P_2(G)\,\triangleleft\,B_2(G)
$$
is described as follows:
\begin{align*}
\mbox{For $r<s<t$,\hspace{1mm}}&\begin{cases}
\{a,(b,c)\}{\cdot}(r,(s,t)){\cdot}\{a,(b,c)\}^{-1}=((b,c),a){\cdot}((s,t),r){\cdot}((b,c),a)^{-1};\\
\{a,(b,c)\}{\cdot}((s,t),r){\cdot}\{a,(b,c)\}^{-1}=(r,(s,t)).\\
\end{cases} \\\rule{0mm}{13mm}
\mbox{For $s<r<t$,\hspace{1mm}}&\begin{cases}
\{a,(b,c)\}{\cdot}(r,(s,t)){\cdot}\{a,(b,c)\}^{-1}=((s,t),r){\cdot}((b,c),a)^{-1};\\
\{a,(b,c)\}{\cdot}((s,t),r){\cdot}\{a,(b,c)\}^{-1}=
\begin{cases}
((b,c),a), & \mbox{if $(r,(s,t))=(a,(b,c))$;}\\
((b,c),a){\cdot}(r,(s,t)), & \mbox{otherwise.}
\end{cases}
\end{cases} \\\rule{0mm}{9mm}
\mbox{For $s<t<r$,\hspace{1mm}}&\begin{cases}
\{a,(b,c)\}{\cdot}(r,(s,t)){\cdot}\{a,(b,c)\}^{-1}=((s,t),r);\\
\{a,(b,c)\}{\cdot}((s,t),r){\cdot}\{a,(b,c)\}^{-1}=((b,c),a){\cdot}(r,(s,t)){\cdot}((b,c),a)^{-1}.\\
\end{cases} 
\end{align*}
\end{corollary}

\begin{proposition}\label{thetaclasificante}
If $G$ is not homeomorphic to an interval, then the morphism $\theta\colon B_2(G)\to\mathbb{Z}_2$ induced in fundamental groups by the classifying map of the double covering $D_2(G)\to\UD_2(G)$ is given on generators by
$$
\theta(\{r,(s,t)\})=\begin{cases} 
1, & \mbox{if \ $s<r<t$;} \\ 0, & \mbox{otherwise.}
\end{cases}
$$
\end{proposition}
\begin{proof}
Since $\theta$ vanishes on $P_2(G)$, the relations
\begin{align*}
(r,(s,t))&=\{r,(s,t)\}, \mbox{ for } r<s<t, \\
((s,t),r)&=\{r,(s,t)\}, \mbox{ for }s<t<r
\end{align*}
force $\theta(\{r,(s,t)\})=0$ provided $r<s$ or $t<r$. On the other hand, for $s<r<t$, the relation $$((s,t),r)=\{r,(s,t)\}\cdot\{a,(b,c)\}$$ gives $\theta(\{r,(s,t)\})=\theta(\{a,(b,c)\})$, which is forced to be 1, since $\theta$ is surjective (recall that $G$ is not an interval, so that $\C_2(G)$ is connected).
\end{proof}

\begin{proposition}\label{proyeccion1}
The morphism $(p_1)_{\#}\colon P_2(G)\to\pi_1(G)$ induced in fundamental groups by the projection $p_1\colon D_2(G)\to G$ onto the first coordinate is trivial on generators $(r,(s,t))$, while
$$
(p_1)_{\#}((s,t),r)=\begin{cases}
(s,t), & \mbox{if $(s,t)$ is a deleted edge;} \\ 1, &\mbox{otherwise.}
\end{cases}
$$
\end{proposition}
\begin{proof}
The conclusion for $(p_1)_{\#}((s,t),r)$ when $r<s<t$ follows from the fact that~(\ref{representingloopexpanded}) is a representing loop for $((s,t),r)\in P_2(G)$, and noticing that the $p_1$-image of a collapsible 1-edge lands in the tree $T$ and, so, it plays no role in $\pi_1(G)$. The other five cases are treated similarly and are left as an exercise for the reader.
\end{proof}

\section{Borsuk-Ulam property}\label{sectiononBUP}
Let $(\Gamma,\tau,G)$ be as in Section~\ref{intro}. We assume in effect all hypotheses, ingredients and constructions set up in Subsection~\ref{secciongbg} around $G$ and its braid groups. Theorem~\ref{BUviatrenzas} below is a ``graph'' version of \cite[Theorem 7]{MR3947929}, proven with the same argument, using that graph configuration spaces are $K(\pi,1)$'s.

\begin{theorem}\label{BUviatrenzas}
The classes $\alpha\in[\Gamma,G]$ for which the Borsuk-Ulam property fails are presicely those fitting on a commutative diagram of groups
\begin{equation}\label{technique}\xymatrix{
 & \pi_1(G) \\
\pi_1(\Gamma) \ar[r]^{\varphi} \ar[ur]^{\alpha_{\#}} \ar@{^{(}->}[d] & P_2(G) \ar[u]_{(p_1)_{\#}} \ar@{^{(}->}[d]\\
\pi_1(\Gamma/\tau) \ar[r]^{\psi} \ar@{>>}[d]_{\theta_1} & B_2(G) \ar@{>>}[d]^{\theta_2}\\
\mathbb{Z}_2 \ar@{=}[r] & \mathbb{Z}_2,
}\end{equation}
for suitable morphisms $\varphi$ and $\psi$. Here the two central group inclusions are induced by the obvious 2-fold covering projections, while morphisms $\theta_i$ are induced by the corresponding classifying maps. In particular, both downward vertical sequences are short exact.
\end{theorem}

\begin{example}\label{casodearboles}{\em
Assume $G=T$, a tree (not homeomorphic to an interval). Then $\theta_2$ is surjective. Since $\pi_1(\Gamma/\tau)$ is free (see paragraph below), it is possible to choose a lifting $\psi\colon\pi_1(\Gamma/\tau)\to B_2(G)$ of~$\theta_1$ along~$\theta_2$. The restricted map $\varphi\colon\pi_1(\Gamma)\to P_2(G)$ then completes diagram~(\ref{technique}) with $\alpha_0\in[\Gamma,G]$ necessarily the unique (trivial) homotopy class. This proves Theorem~\ref{maintheorem} when $G$ is contractible. Therefore, throughout the rest of the section we assume that $G\neq T$. In particular $\pi_1(G)=F(z_1,\ldots,z_k)$, the free group on generators $z_i=(x_i,y_i)$, where $\{(x_i,y_i)\}_{i=1,\ldots,k}$ is the set of deleted edges of $G$ (recall $x_i<y_i$ for all $i$, which gives the orientation of the representing loop for $z_i$ ---after collapsing~$T$ to a point). For convenience, we will assume that the deleted edges have been arranged so that $y_1<y_2<\ldots<y_k$.
}\end{example}

Ignoring vertices of degree 2, $\tau$ is forced to act at the level of vertices and edges. Thus $\Gamma/\tau$ has a natural graph structure. A simple Euler characteristic argument then show that the free groups $\pi_1(\Gamma)$ and $\pi_1(\Gamma/\tau)$ have respective ranks $2m+1$ and $m+1$, where $m=-\chi(\Gamma)/2\geq0$. In particular,
\begin{equation}\label{elts}
[\Gamma,G]=F(z_1,\ldots,z_k)^{2m+1}.
\end{equation}
More explicitly:

\begin{lemma}\label{generatorsviaranks}
It is possible to choose generators $a,a_1,a'_1,a_2,a'_2,\ldots,a_m,a'_m$ of $\pi_1(\Gamma)$ as well as generators $c,c_1,c_2,\ldots,c_m$ of $\pi_1(\Gamma/\tau)$ satisfying
$$\mbox{
$a=c^2$, $a_i=c_i$, $a'_i=cc_ic^{-1}$, $\theta_1(c)=1\in\mathbb{Z}_2$, and $\theta_1(c_i)=0\in\mathbb{Z}_2$
}$$
for $i=1,2,\ldots,m$. In this setting, $\alpha\in[\Gamma,G]$ is identified under~(\ref{elts}) with the tuple $$(\alpha_{\#}(a),\alpha_{\#}(a_1),\alpha_{\#}(a'_1),\ldots,\alpha_{\#}(a_m),\alpha_{\#}(a'_m)).$$
\end{lemma}

\begin{proof}
This lemma is part of the folklore. Since we cannot find an explicit reference, we sketch a proof. Suppose 
we have an epimorphism  $\theta:\pi_1(\Gamma/\tau) \to  \mathbb{Z}_2$ and let $\{e_0,e_1,...,e_m\}$ be an arbitrary base. Assume without loos of generality that $\theta(e_0)=1\in\mathbb{Z}_2$ and set $c:=e_0$. If $\theta(e_i)=0$ for $i>0$,  then  
set $c_i:=e_i$. Otherwise, let $c_i:=e_0\cdot e_i$. Then we have constructed a base $\{c, c_1,...,c_m\}$ with the desired $\theta$-properties. Now we apply the 
Reidemeister-Schreier process to find a presentation of $\ker(\theta)$, using the set of generators ${c, c_1,...,c_m}$ and, as Schreier 
system, $\{1, c\}$. It follows that the kernel has a presentation given by elements as in the statement of the lemma subject to  
no relation.  The result follows. 
\end{proof}

\begin{proof}[Proof of Theorem~\ref{secondmain}]
 It is a standard fact that the right-hand side column in~(\ref{technique}) becomes
 \begin{equation}\label{babyexample} \xymatrix{
\pi_1(G)=\mathbb{Z} & &\mathbb{Z}=P_2(S^1) \ar[ll]_{p_1=\text{\hspace{.3mm}Id}} \ar@{^{(}->}[r]^2 & B_2(S^1)=\mathbb{Z} \ar[rrr]^{\hspace{6mm}\theta_2=\text{\hspace{.3mm}mod-2 proj}} & & & \mathbb{Z}_2.
}\end{equation}
Consequently, morphisms $\psi\colon\pi_1(\Gamma/\tau)\to B_2(G)$ satisfying $\theta_2\circ\psi=\theta_1$ are in one-to-one correspondence with tuples of integer numbers $(p,q_1,\ldots,q_m)$, with $p$ odd and each $q_i$ even, where the correspondence is so that  $\psi(c)=p$ and $\psi(c_i)=q_i$ for $1\leq i\leq m$. Say $q_i=2p_i$. Lemma~\ref{generatorsviaranks} and~(\ref{babyexample}) then imply that the restriction to $\pi_1(\Gamma)$ of such a $\psi$ is given by $\varphi(c)=p$ and $\varphi(a_i)=p_i=\varphi(a'_i)$ for $i=1,2,\ldots,m$. The result follows.
\end{proof}

The proof of Theorem~\ref{maintheorem} follows the strategy in the previous proof, except that the needed algebraic manipulations are far much subtler, and depend on the results in Subsetion~\ref{secciongbg}. With this in mind, we assume from this point on that $G$ is not homeomorphic to a circle (or to a tree, in view of Example~\ref{casodearboles}), and pick key elements $\rho,\lambda_1,\lambda_2,\ldots,\lambda_k\in P_2(G)$ and $\sigma\in B_2(G)$ as described below. We then set $\lambda'_i:=\sigma\lambda_i\hspace{.3mm}\sigma^{-1}\in P_2(G)$.

\begin{figure}[h!]
\centering
\begin{tikzpicture}
\node[style={circle, draw, scale=.6}, scale=1.0, xscale = -1] (1) at ( 1.3, 3) {0};
\node[scale=.2] (2) at ( 3.06, 3) {};
\node[scale=.2] (3) at ( 4.53, 3) {};
\node[style={circle, draw, scale=.7}] (4) at ( 6.0, 3) {$v$};
\node[style={circle, draw, scale=.6}] (5) at ( 6.6, 2.4) {$v_2$};
\node[style={circle, draw, scale=.6}] (12) at ( 6.6, 3.6) {$v_1$};
\node[style={circle, draw, scale=.6}] (13) at ( 2.3, 3 ) {$1$};
\draw[very thin] (13) -- (1); \draw[very thin] (13) -- (2);
\draw[dashed, very thin] (3) -- (2);
\draw[very thin] (4) -- (3);
\draw[very thin] (5) -- (4);
\draw[very thin] (12) -- (4);
\end{tikzpicture}
\caption{Part of the non-linear tree $T$ showing the essential vertex $v$}
\label{ing}
\end{figure}
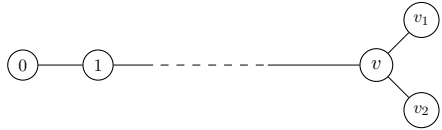

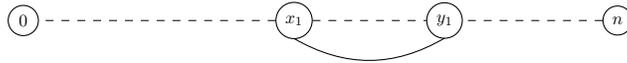
\begin{figure}[h!]
\centering
\begin{tikzpicture}
\node[style={circle, draw, scale=.6}, scale=1.0, xscale = -1] (1) at ( -1.3, 3) {0};
\node[scale=.2] (2) at ( 3.06, 3) {};
\node[scale=.2] (3) at ( 4.53, 3) {};
\node[style={circle, draw, scale=.6}] (4) at ( 4.3, 3) {$y_1$};
\node[style={circle, draw, scale=.6}] (5) at ( 6.6, 3) {$n$};
\node[style={circle, draw, scale=.6}] (13) at ( 2.3, 3 ) {$x_1$};
\draw[dashed,very thin] (13) -- (1);
\draw[dashed, very thin] (13) -- (4);
\draw[dashed, very thin] (5) -- (4);
\path[line width=.14mm] (2.3, 2.76) edge [bend right] (4.3,2.765);
\end{tikzpicture}
\caption{The linear tree $T$ together with the erased edge $(x_1,y_1)$}
\label{ong}
\end{figure}

\noindent  {\bf When $T$ has an essential vertex $v$}. Choose vertices $v_1,v_2$ with $v<v_1<v_2$ and $\tau(e_{v_1})=v=\tau(e_{v_2})$. See Figure~\ref{ing}. In this situation we choose the critical 1-cell connecting the trees $DT_d$ and $DT_u$ to be $(a,(b,c)):=(v_1,(v,v_2))$ (see Subsection~\ref{secciongbg}), and set
$$ 
\sigma:=\{v_1,(v,v_2)\},\quad\rho:=((v,v_2),v_1)\quad\mbox{and}\quad\lambda_i:=(x_i+1,(x_i,y_i)),
$$ 
for $1\leq i\leq k$ (note that $x_i+1<y_i$ since $G$ is simplicial). Corollary~\ref{conjugacion} and Propositions~\ref{inclusion} and~\ref{proyeccion1} then give
\begin{equation}\label{propertiesnolinear}
\mbox{$\rho=\sigma^2$, \ $p_1(\rho)=1$, \ $p_1(\lambda_i)=1$ \ and \ $p_1(\lambda'_i)=z_i$, \ for $i=1,2,\ldots,k$.}
\end{equation}
Here and below, we write $p_1$ instead of $(p_1)_{\#}$. The context clarifies the abuse of notation.

\noindent  {\bf When $T$ is linear}. In this situation we choose the critical 1-cell connecting the trees $DT_d$ and $DT_u$ to be $(a,(b,c)):=(x_1+1,(x_1,y_1))$. See Figure~\ref{ong}. Recall we have chosen $y_1<y_i$ for $i>1$. Furthermore, the condition $x_1+1<y_1$ is forced because $G$ is simplicial. Then set
$$ 
\sigma:=\{x_1+1,(x_1,y_1)\},\ \ \rho:=((x_1,y_1),x_1+1),\ \ \lambda_1:=(x'_1,(x_1,y_1))\ \ \mbox{and} \ \ \lambda_i:=(x_i+1,(x_i,y_i)),
$$ 
for $2\leq i\leq k$, where $$x'_1=\begin{cases}0,&\mbox{if $x_1>0$;}\\y_1+1,&\mbox{if $x_1=0$ (so $y_1<n$, even if $k=1$, since $G$ is not homeomorphic to a circle).}\end{cases}$$ Corollary~\ref{conjugacion} and Propositions~\ref{inclusion} and~\ref{proyeccion1} then give
\begin{equation}\label{propertieslinear}
\mbox{$\rho=\sigma^2$, \ $p_1(\rho)=z_1$, \ $p_1(\lambda_i)=1$, \ $p_1(\lambda'_1)=z_1$ \ and \ $p_1(\lambda'_i)=z_iz^{-1}_1$ \ for $i=2,3,\ldots,k$.}
\end{equation}

\begin{proof}[Proof of Theorem~\ref{maintheorem}]
Let $\lambda$, $\lambda'$ and $z$ denote, respectively, the sequences of symbols $(\lambda_1,\ldots,\lambda_k)$, $(\lambda'_1,\ldots,\lambda'_k)$ and $(z_1,\ldots,z_k)$, and consider an arbitrary sequence of words
\begin{equation}\label{word}
\left(\rule{0mm}{4mm}w(z),w_1(z),w'_1(z),\ldots,w_m(z),w'_m(z)\right)\in F(z_1,\ldots,z_k)^{2m+1}.
\end{equation}
We have to prove that the homotopy class $\alpha\in[\Gamma,G]$ corresponding to~(\ref{word}) under~(\ref{elts}) and Lemma~\ref{generatorsviaranks} fits in a commutative diagram~(\ref{technique}) for suitable morphisms $\psi$ and $\varphi$.

\medskip
Assume $T$ has an essential vertex $v$ and consider the setup in~(\ref{propertiesnolinear}). The map $\psi\colon\pi_1(\Gamma/\tau)\to B_2(G)$ defined by
$$
\psi(c)=w(\lambda)\sigma \mbox{ \ \ and \ \ } \psi(c_i)=\sigma w_i(\lambda)\sigma^{-1}w'_i(\lambda), \ \ i=1,\ldots,m,
$$
satisfies $\theta_1=\theta_2\circ\psi$, in view of Proposition~\ref{thetaclasificante} and Lemma~\ref{generatorsviaranks}. Furthermore, the restricted map $\varphi\colon\pi_1(\Gamma)\to P_2(G)$ satisfies
\begin{align*}
\varphi(a)&=\psi(c)^2=w(\lambda)\sigma w(\lambda)\sigma=w(\lambda)\cdot\sigma w(\lambda)\sigma^{-1}\cdot \rho,\\
\varphi(a_i)&=\psi(c_i)=\sigma w_i(\lambda)\sigma^{-1} \cdot w'_i(\lambda),\\
\varphi(a'_i)&=\psi(c) \psi(c_i) \psi(c)^{-1}=w(\lambda)\sigma\sigma w_i(\lambda)\sigma^{-1}w'_i(\lambda) \sigma^{-1} w(\lambda)^{-1}\\&=w(\lambda)\rho w_i(\lambda)\rho^{-1}\cdot \sigma w'_i(\lambda) \sigma^{-1} \cdot w(\lambda)^{-1}.
\end{align*}
So
\begin{align*}
p_1(\varphi(a))&=p_1(w(\lambda))\cdot p_1(\sigma w(\lambda)\sigma^{-1})\cdot p_1(\rho)=p_1(\sigma w(\lambda)\sigma^{-1})=p_1(w(\lambda'))=w(p_1(\lambda'))=w(z),\\
p_1(\varphi(a_i))&=p_1(\sigma w_i(\lambda)\sigma^{-1})\cdot p_1(w'_i(\lambda))=p_1(w_i(\lambda'))=w_i(p_1(\lambda'))=w_i(z),\\
p_1(\varphi(a'_i))&=p_1(w(\lambda)\rho w_i(\lambda)\rho^{-1})\cdot p_1(\sigma w'_i(\lambda) \sigma^{-1})\cdot p_1(w(\lambda))^{-1}=p_1(\sigma w'_i(\lambda) \sigma^{-1})=p_1(w'_i(\lambda'))=w'_i(z),
\end{align*}
which yields the result in the case under consideration.

\medskip
Assume now that $T$ is linear and consider the setup in~(\ref{propertieslinear}). Write the elements $w(z)z_1^{-1}$, $w_i(z)z_1^{-1}$ and $w'_i(z)z_1^{-1}$ of $F(z_1,z_2,\ldots,z_k)$ as words on the generators $t_1,t_2,\ldots,t_k$ of $F(z_1,z_2,\ldots,z_k)$ given by $t_1:=z_1$ and $t_i:=z_iz_1^{-1}$ for $i\geq2$. Say
$$w(z)z_1^{-1}=\ell(t),\quad w_i(z)z_1^{-1}=\ell_i(t),\quad w'_i(z)z_1^{-1}=\ell_i\hspace{-.9mm}'(t),$$
where $t$ stands for the tuple $(t_1,t_2,\ldots,t_k)$. The map $\psi\colon\pi_1(\Gamma/\tau)\to B_2(G)$ defined by
$$
\psi(c)=\ell(\lambda)\sigma \mbox{ \ \ and \ \ } \psi(c_i)=\sigma \ell_i(\lambda)\sigma \lambda_1^{-1} \ell_i\hspace{-.9mm}'(\lambda)\lambda_1, \ \ i=1,\ldots,m,
$$
satisfies $\theta_1=\theta_2\circ\psi$, in view of Proposition~\ref{thetaclasificante} and Lemma~\ref{generatorsviaranks}. Furthermore, the restricted map $\varphi\colon\pi_1(\Gamma)\to P_2(G)$ satisfies
\begin{align*}
\varphi(a)&=\psi(c)^2=\ell(\lambda)\sigma \ell(\lambda)\sigma=\ell(\lambda)\cdot\sigma \ell(\lambda)\sigma^{-1}\cdot\rho,\\
\varphi(a_i)&=\psi(c_i)=\sigma \ell_i(\lambda)\sigma\lambda_1^{-1} \ell_i\hspace{-.9mm}'(\lambda)\lambda_1
=\sigma \ell_i(\lambda)\sigma^{-1}\cdot\rho\lambda_1^{-1} \ell_i\hspace{-.9mm}'(\lambda)\lambda_1,\\
\varphi(a'_i)&=\psi(c) \psi(c_i) \psi(c)^{-1}=\ell(\lambda)\sigma\cdot \sigma \ell_i(\lambda)\sigma\lambda_1^{-1} \ell_i\hspace{-.9mm}'(\lambda) \lambda_1 \cdot \sigma^{-1} \ell(\lambda)^{-1}\\
&=\ell(\lambda)\rho \ell_i(\lambda)\cdot\sigma \lambda_1^{-1} \ell_i\hspace{-.9mm}'(\lambda) \lambda_1 \sigma^{-1} \cdot\ell(\lambda)^{-1}.
\end{align*}
So
\begin{align*}
p_1(\varphi(a))&=p_1(\ell(\lambda))\cdot p_1(\sigma \ell(\lambda)\sigma^{-1})\cdot p_1(\rho)=p_1(\ell(\lambda')) z_1=\ell(p_1(\lambda'))z_1=\ell(t)z_1=w(z),\\
p_1(\varphi(a_i))&=p_1(\sigma \ell_i(\lambda)\sigma^{-1})\cdot p_1(\rho\lambda_1^{-1} \ell_i\hspace{-.9mm}'(\lambda)\lambda_1)=p_1(\sigma \ell_i(\lambda)\sigma^{-1})z_1=p_1(\ell_i(\lambda'))z_1=\ell_i(t)z_1=w_i(z),\\
p_1(\varphi(a'_i))&=p_1(\ell(\lambda)\rho \ell_i(\lambda))\cdot p_1(\sigma\lambda_1^{-1} \ell_i\hspace{-.9mm}'(\lambda)\lambda_1 \sigma^{-1})\cdot p_1(\ell(\lambda))^{-1}=z_1p_1(\sigma\lambda_1^{-1} \ell_i\hspace{-.9mm}'(\lambda)\lambda_1 \sigma^{-1})\\&=z_1p_1((\lambda'_1)^{-1} 
\ell_i\hspace{-.9mm}'(\lambda')\lambda'_1)=z_1(p_1(\lambda'_1))^{-1}\cdot 
p_1(\ell_i\hspace{-.9mm}'(\lambda'))\cdot p_1(\lambda'_1)=z_1z_1^{-1}\ell_i\hspace{-.9mm}'(t)z_1=w'_i(z),
\end{align*}
which completes the proof. 
 \end{proof}
 

{\sc \ 

Departamento de Matem\'atica, IME

Universidade de S\~ao Paulo

Rua do Mat\~ao 1010 CEP: 05508-090

S\~ao Paulo-SP, Brazil

{\tt dlgoncal@ime.usp.br}

\bigskip

Departamento de Matem\'aticas

Centro de Investigaci\'on y de Estudios Avanzados del I.P.N.

Av.~Instituto Polit\'ecnico Nacional n\'umero~2508, San Pedro Zacatenco

M\'exico City 07000, M\'exico.}

{\tt jesus@math.cinvestav.mx}

\end{document}